\documentclass[11pt,reqno]{amsart}
\usepackage{amssymb}
\usepackage{mathdots}

\usepackage{mathrsfs}
\usepackage{amsthm}
\usepackage{multirow}
\usepackage{nameref}
\newtheorem{theorem}{Theorem}[section]

\newtheorem{corollary}[theorem]{Corollary}

\theoremstyle{definition}

\theoremstyle{remark}

\numberwithin{equation}{section}

\newcommand{\abs}[1]{\lvert#1\rvert}
\newcommand\myeq{\mathrel{\overset{\makebox[0pt]{\mbox{\normalfont\tiny\sffamily def}}}{=}}}

\begin{document}

\title{ESTIMATES ON SINGULAR VALUES OF FUNCTIONS OF PERTURBED OPERATORS }
%\author[Affil1]{Qinbo Liu}

%\author[Affil2]{V.V. Peller\corref{cor1}}
\author{Qinbo Liu\\
Department of Mathematics, Michigan State University\\
East Lansing, MI 48824, USA}
\thanks{Corresponding Author:\\
\indent \textit{Email address:} liuqinbo@msu.edu(Qinbo Liu).}\noindent
\date{April 25, 2016}

\begin{abstract}
%% Text of abstract
This is a conitunation of \cite{A} and \cite{B}. We prove that if function $f$ belongs to the class
 $\Lambda_{\omega} \myeq \{f: \omega_{f}(\delta)\leq \text{const} \; \omega(\delta)\} $ for an arbitrary modulus of 
 continuity $\omega$,
 then
%the following inequality holds for all 
$ s_j(f(A)-f(B))\leq c\cdot \omega_{\ast}\big((1+j)^{-\frac{1}{p}}\Vert A-B \Vert_{S_{p}^l}\big)  \cdot \Vert f \Vert_{\Lambda_{\omega}}
$ for arbitrary self-adjoint operators $A$, $B$ and all $1\leq j\leq l$, where 
 $\omega_{\ast}(x) \myeq x \int_{x}^{\infty}\frac{\omega(t)}{t^2}dt \;( x>0) $. The result is then generalized for contractions, maximal dissipative operators, normal operators and $n$-tuples of commuting self-adjoint operators.
\end{abstract}
\maketitle
\tableofcontents

%\begin{keyword}
%% keywords here, in the form: keyword \sep keyword.  You may use no more than 10 keywords.
%Keyword 1 \sep Keyword 2 \sep Keyword 3
%\end{keyword}

%% Do not remove the page break here.
%\pagebreak

%\maketitle
%\footnotetext[1]{CCC}
%\addtocounter{footnote}{1}
\section{Introduction}
%\section[Introduction]{Introduction\protect\footnote{hi}}
In this note we study the behavior of functions of operators
 under perturbations. We are going to find estimates for the singular values $s_n(f(A)-f(B))$, 
 where both $A$ and $B$ are arbitrary self-adjoint or unitary operators.
 These results are based on the methods developed in \cite{A} and \cite{C} for estimates of operator norms
 $ \Vert f(A)-f(B) \Vert$, in these papers the authors proved if $f$ belongs to the H\"older class
  $\Lambda_{\alpha}(\mathbb{R})$
 with $0<\alpha<1$, then  $ \Vert f(A)-f(B) \Vert \leq \text{const}\;\Vert f \Vert_{\Lambda_{\alpha}} \Vert A-B \Vert^{\alpha}$ for 
 all pairs of self-adjoint or unitary operators $A$ and $B$. The authors also generalized their results to the class $\Lambda_{\omega}$,
  and obtained estimate $ \Vert f(A)-f(B) \Vert \leq \text{const}\;\Vert f \Vert_{\Lambda_{\omega}} \omega_{\ast} \Vert A-B \Vert$.
  
  In \cite{B}, it was shown that for functions $f$ in the H\"older class
  $\Lambda_{\alpha}(\mathbb{R})$ with $0<\alpha<1$ and if $1<p<\infty$, the operator $f(A)-f(B)$ belongs to $\mathbf{S}_{p/\alpha}$, whenever $A$ and $B$ are arbitrary self-adjoint operators such that $A-B \in \mathbf{S}_p$. In particular, it was proved that if $0< \alpha < 1$, then there exists a constant
  $c>0$ such that for every $l \geq 0$, $p \in [1, \infty)$, $f \in \Lambda_{\alpha}(\mathbb{R})$, and for arbitrary self-adjoint operators 
  $A$ and $B$ on Hilbert space with bounded $A-B$, the following inequality holds for every $j \leq l$:
  \begin{equation*}
   s_j(f(A)-f(B))\leq c \; \Vert f \Vert_{\Lambda_{\alpha}(\mathbb{R})}
  (1+j)^{-\frac{\alpha}{p}}\Vert A-B \Vert_{S_{p}^l}^{\alpha}(\text{see (\ref{Nor:spl}) for definition}).
   \end{equation*}
  In section \S\ref{sec:num2}, we generalize this estimate to the class $\Lambda_{\omega}$  and also obtain some lower-bound estimates for rank one perturbations which also extend the results in \cite{B}. 
  In section \S\ref{sec:num3}, similar estimates are given without proofs in case of contractions, maximal dissipative operators, normal operators and $n$-tuples of commuting self-adjoint operators.
  
  Necessary information on Space $\Lambda_{\omega}$ is given in section \S \ref{sec:num1}. We refer the reader to \cite{A} for more detailed information.

%% The correct journal style for \specialsection is all uppercase; a known bug
%% in amsart.cls prevents this, so input must be uppercase until it is fixed.
%\specialsection*{This is a Special Section Head}
%\specialsection*{THIS IS A SPECIAL SECTION HEAD}
%This is an example of a special section head%
%%%%%%%%%%%%%%%%%%%%%%%%%%%%%%%%%%%%%%%%%%%%%%%%%%%%%%%%%%%%%%%%%%%%%%%%
%\footnote{Here is an example of a footnote. Notice that this footnote
%text is running on so that it can stand as an example of how a footnote
%with separate paragraphs should be written.
%\par
%And here is the beginning of the second paragraph.}%
%%%%%%%%%%%%%%%%%%%%%%%%%%%%%%%%%%%%%%%%%%%%%%%%%%%%%%%%%%%%%%%%%%%%%%%%
\section{Space $\Lambda_{\omega}$}
\label{sec:num1}
Let $\omega$ be a modulus of continuity, i.e., $\omega$ is a nondecreasing
continuous function on $[0, \infty)$ such that $\omega(0)=0$, $\omega(x)>0$ 
for $x>0$, and
\begin{equation*}
  \omega(x+y)\leq \omega(x)+\omega(y), x, y\in [0,\infty).
\end{equation*} 
We denote by $\Lambda_{\omega}(\mathbb{R})$ the space of functions on $\mathbb{R}$
such that
\begin{equation*}
\Vert f \Vert_{\Lambda_{\omega}(\mathbb{R})} \myeq \sup_{x \neq y}\frac{\abs{f(x)-f(y)}}{\omega(\abs{x-y})}.
\end{equation*}

The space  $\Lambda_{\omega}(\mathbb{T})$ on the unit circle can be defined in a similar way.

We continue with the class $\Lambda_{\omega}$ of functions on $\mathbb{T}$ first.
 Let $w$ be an infinitely differentiable function on $\mathbb{R}$ such that
\begin{equation}
  w \geq 0, \text{supp}\; w \subset [\frac{1}{2}, 2], \;\text{and}\; w(x)=1-w(\frac{x}{2})
 \; \text{for}\; x \in [1,2].
 \label{Dfn:doubleu}
\end{equation}
Define a $C^{\infty}$ function $v$ on $\mathbb{R}$ by 
\begin{equation}
v(x)=1 \; \text{for}\; x \in[-1,1]\; \text{and}\; v(x)=w(\abs{x}) \; \text{if} \; \abs{x}\geq 1.
\label{Dfn:smallv}
\end{equation}
Define trigonometric polynomials $W_n$, $W_{n}^{\sharp}$ and $V_n$ by
\begin{equation*}
  W_n(z)=\sum_{k\in \mathbb{Z}}w(\frac{k}{2^n})z^k,  n\geq 1, \; W_0(z)=\bar{z}+1+z,
  \; \text{and}\; W_{n}^{\sharp}(z)=\overline{W_n(z)}, n\geq 0
\end{equation*}
and
\begin{equation*}
V_n(z)=\sum_{k\in \mathbb{Z}}w(\frac{k}{2^n})z^k,  n\geq 1.
\end{equation*}
$V_n$ is called de la Vall\'{e}e Poussin type kernel.

If $f$ is a distribution on $\mathbb{T}$, we define $f_n$, $n\geq 0$ by
\begin{equation*}
f_n=f*W_n+f*W_{n}^{\sharp}, \; n\geq 1, \; \text{and} \; f_0=f*W_0,
\end{equation*}
Then $f=\sum_{n \geq 0}f_n$ and $f-f*V_n=\sum_{k=n+1}^{\infty}f_n$.

Now we proceed to the real line case. We use the same functions $w$, $v$ as in (\ref{Dfn:doubleu}), (\ref{Dfn:smallv})
, and define functions $W_n$, $W_{n}^{\sharp}$ and $V_n$ on $\mathbb{R}$ by
\begin{equation*}
  \mathcal{F}W_n(x)=w(\frac{x}{2^n}), \; \mathcal{F}W_n^{\sharp}(x)=\mathcal{F}W_n(-x), \; n \in \mathbb{Z}
  \end{equation*}
and
\begin{equation*}
\mathcal{F}V_n(x)=v(\frac{x}{2^n}),  \; n \in \mathbb{Z},
\end{equation*}
where $\mathcal{F}$ is the \textit{Fourier transform}:
\begin{equation*}
(\mathcal{F}f)(t)=\int_{\mathbb{R}}f(x)e^{-ixt}dx, \; f \in L^1.
\end{equation*} 
$V_n$ is also called de la Vall\'{e}e Poussin type kernel.

If $f$ is a tempered distribution on $\mathbb{R}$, we define $f_n$  by
\begin{equation*}
f_n=f*W_n+f*W_{n}^{\sharp}, \; n\in \mathbb{Z}.
\end{equation*}

We will use the same notation $\Lambda_{\omega}$, $W_n$, $W_{n}^{\sharp}$ and $V_n$ on $\mathbb{R}$ and on $\mathbb{T}$ in the following discussion.

In \cite{A}, it is proved that there exists a constant $c$ such that for an arbitrary modulus of
continuity $\omega$ and for an arbitrary function $f$ in $\Lambda_{\omega}$, the following inequalities hold for
all $n \in \mathbb{Z}$, in $\mathbb{R}$ case, or for all $n \geq 0$, in $\mathbb{T}$ case:
\begin{equation}
\Vert f-f*V_n \Vert_{L^{\infty}} \leq c \; \omega(2^{-n}) \Vert f \Vert_{\Lambda_{\omega}}
\label{Ine:valleeinf} 
\end{equation}
\begin{equation}
\Vert f*W_n \Vert_{L^{\infty}} \leq c \; \omega(2^{-n}) \Vert f \Vert_{\Lambda_{\omega}}, \;
\Vert f*W_n^{\sharp} \Vert_{L^{\infty}} \leq c \; \omega(2^{-n}) \Vert f \Vert_{\Lambda_{\omega}}
\label{Ine:doubleUninf}  
\end{equation}
%\begin{equation}
%\Vert f*W_n^{\sharp} \Vert_{L^{\infty}} \leq c \; \omega(2^{-n}) \Vert f \Vert_{\Lambda_{\omega}} 
%\label{Ine:doubleUninf2}
%\end{equation}

Let $\mathscr{S}'(\mathbb{R})$
 be the space of all tempered distributions on  $\mathbb{R}$. Denote by $\mathscr{S}'_{+}(\mathbb{R})$
 the set of all $f \in \mathscr{S}'(\mathbb{R})$ such that $\text{supp }\mathcal{F}f \subset [0, \infty)$.
 Put $\big(\Lambda_{\omega}(\mathbb{R})\big)_{+} \myeq \Lambda_{\omega}(\mathbb{R})\cap \mathscr{S}'_{+}(\mathbb{R})$ and $\mathbb{C}_{+} \myeq \{z \in \mathbb{C}: \text{Im } z >0 \}$. Then a function in $\Lambda_{\omega}(\mathbb{R})$ belongs to the space $\big(\Lambda_{\omega}(\mathbb{R})\big)_{+}$ if and only if 
 it has a (unique) continuous extension to the closed upper half-plance $\text{clos } \mathbb{C}_{+}$ that is analytic in the open upper half-plane $\mathbb{C}_{+}$ with at most a polynomial growth rate at infinity.

\section{Estimates on  singular values of functions of perturbed self-adjoint and unitary operators}
\label{sec:num2}
Recall that if $T$ is a bounded linear operator on Hilbert space, then the singular values 
$s_j(T)$, $j \geq 0$, are defined by
 \[s_j(T)= \inf \{\Vert T-R \Vert: \; \text{rank} R \leq j\}.\]
For $l \geq 0$ and $p \geq 1$, we consider the norm $S_p^{l}$ (see \cite{I}) defined by
\begin{equation}
\Vert T \Vert_{S_p^{l}} \myeq \big( \sum_{j=0}^{l}(s_j(T))^p \big)^{\frac{1}{p}}.
\label{Nor:spl}
\end{equation}
It is shown in \cite{M} and \cite{B} that if $f$ is an entire function of exponential type at most $\sigma$
that is bounded on $\mathbb{R}$, and $A$, $B$ are self-adjoint operators with bounded $A-B$, then
\begin{equation}
\Vert f(A)-f(B) \Vert \leq \text{const}\; \sigma \Vert f \Vert_{L^{\infty}} \Vert A-B \Vert,
\label{Ine:opnorm}
\end{equation}
and
\begin{equation}
\Vert f(A)-f(B) \Vert_{S_p^l} \leq \text{const}\; \sigma \Vert f \Vert_{L^{\infty}} \Vert A-B \Vert_{S_p^l}.
\label{Ine:plnorm}
\end{equation}  
For the proof and more details, see \cite{A}, \cite{B}, \cite{H}, \cite{J}, \cite{K} and \cite{M}. 

Given a modulus of continuity $\omega$, define functions $\omega_{\ast}$ and $\omega_{\sharp}$ by
\begin{equation*}
\omega_{\ast}(x) = x \int_{x}^{\infty}\frac{\omega(t)}{t^2}dt \; ,\; x>0
\end{equation*}
and
\begin{equation*}
\omega_{\sharp}(x) = x \int_{x}^{\infty}\frac{\omega(t)}{t^2}dt \;+ \int_{0}^{x}\frac{\omega(t)}{t}dt \; ,\; x>0 .
\end{equation*}
In this paper, we assume that $\omega_{\sharp}$ is finite valued whenever it is used. \\
For example, if we define $\omega$ by 
\begin{equation*} \omega(x)=x^{\alpha},\; x>0, \;0<\alpha<1, 
\end{equation*}
then $\omega_{\sharp}(x) \leq \textnormal{const }\omega(x)$.\\
It is well known(see \cite{F}, Ch.3, Theorem 13.30) that if $\omega$ is a modulus of continuity, then the Hilbert transform maps $\Lambda_{\omega}$ into itself
if and only if $\omega_{\sharp}(x) \leq \textnormal{const }\omega(x)$.

\begin{theorem}
There exists a constant $c>0$ such that for every modulus of continuity $\omega$, every $f$ in
$\Lambda_{\omega}(\mathbb{R})$ and for arbitrary self-adjoint operators $A$ and $B$,
the following inequality holds for all $l$ and for all $j$, $1\leq j\leq l :$
\begin{equation}
   s_j(f(A)-f(B))\leq c\cdot
   \omega_{\ast}\big((1+j)^{-\frac{1}{p}}\Vert A-B \Vert_{S_{p}^l}\big)
   \cdot \Vert f \Vert_{\Lambda_{\omega}}.
   \label{Ine:svalue}
\end{equation}
\label{Thm:mainthm}
\end{theorem}
\begin{proof}
  $A$ and $B$ can be taken as bounded operators(see \cite{C}, Lemma 4.4), then we may
  further assume $f$ is bounded. 
  Let $R_N=\sum_{n=-\infty}^{N}(f_n(A)-f_n(B))$, $Q_N=(f-f*V_N)(A)-(f-f*V_N)(B)$. Here $f_n$
  and the de la Vall\'ee Pouss\'in type kernel $V_N$ are defined as in \S\ref{sec:num1}.
  Then $f(A)-f(B)=R_N+Q_N$, with convergence in the uniform operator topology as shown in \cite{A}.
  Note that for any integer $m \in \mathbb{Z}$, functions $f_m$ and $f-f*V_m$ are entire functions
  of exponential type at most $2^{m+1}.$ Thus it follows from (\ref{Ine:opnorm}),
   (\ref{Ine:plnorm}), (\ref{Ine:valleeinf}), and (\ref{Ine:doubleUninf}) that
  \[
    \Vert Q_N\Vert\leq c\cdot \omega(2^{-N})
  \cdot \Vert f \Vert_{\Lambda_{\omega}},
  \]  
  and
  \begin{align*}
  \Vert R_N\Vert_{S_{p}^l} 
  & \leq \sum_{n=-\infty}^{N}\Vert f_n(A)-f_n(B)\Vert_{S_{p}^l}\\
  & \leq c\cdot  \sum_{n=-\infty}^{N}\big ( 2^n \cdot \Vert f_n \Vert_{L^{\infty}} \big ) \cdot \Vert A-B\Vert_{S_{p}^l}\\
  & \leq c\cdot 2^N \cdot \omega_{\ast}(2^{-N})
  \cdot \Vert A-B\Vert_{S_{p}^l} \cdot \Vert f \Vert_{\Lambda_{\omega}} \text{(see \cite{A})}
  \end{align*}
Then 
 % \begin{IEEEeqnarray*}{rCl}
%  \begin{equation*}
  \begin{eqnarray*}
 % \begin{split}
    s_j(f(A)-f(B)) &\leq & s_j(R_N)+\Vert Q_N \Vert \leq (1+j)^{-1/p}
    \cdot \Vert R_N\Vert_{S_{p}^l}+\Vert Q_N \Vert \\
    &\leq & c\cdot \big[(1+j)^{-\frac{1}{p}}\cdot 2^N 
    \cdot \omega_{\ast}(2^{-N})\Vert A-B \Vert_{S_{p}^l}+
    \omega(2^{-N})\big]\\
   &&\cdot \Vert f \Vert_{\Lambda_{\omega}}
 % \end{split} 
  \end{eqnarray*}
  %\end{equation*}
 %\end{IEEEeqnarray*}
Take $N$ such that $1\leq (1+j)^{-\frac{1}{p}}
\cdot 2^N\cdot \Vert A-B \Vert_{S_{p}^l}\leq 2$ and use the fact that
$\omega(t)\leq\omega_{\ast}(t)$ for any $t>0$, we get (\ref{Ine:svalue}).
\end{proof}

\begin{theorem}
There exists a constant $c>0$ such that for every modulus of continuity $\omega$, every $f$ in
$\Lambda_{\omega}(\mathbb{T})$ and for arbitrary unitary operators $U$ and $V$,
the following inequality holds for all $l$ and for all $j$, $1\leq j\leq l :$
\begin{equation}
   s_j(f(U)-f(V))\leq c\cdot
   \omega_{\ast}\big((1+j)^{-\frac{1}{p}}\Vert U-V \Vert_{S_{p}^l}\big)
   \cdot \Vert f \Vert_{\Lambda_{\omega}}.
   \label{Ine:svalueU}
\end{equation}
\end{theorem}
\begin{proof}
If $(1+j)^{-\frac{1}{p}}
\cdot \Vert U-V \Vert_{S_{p}^l} \leq 2$, the proof is similar to Theorem \ref{Thm:mainthm}
with $R_N=\sum_{n=0}^{N}(f_n(U)-f_n(U))$; 
if $(1+j)^{-\frac{1}{p}}
\cdot \Vert U-V \Vert_{S_{p}^l}>2$, then
\[
 s_j(f(U)-f(V))\leq  \Vert f(U)-f(V)\Vert\leq
 c\cdot \omega_{\ast}(\Vert U-V\Vert)\cdot
 \Vert f \Vert_{\Lambda_{\omega}}
 \leq  c\cdot \omega_{\ast}(2)\cdot
 \Vert f \Vert_{\Lambda_{\omega}}. 
\]
\end{proof}

\begin{corollary}
Let $\omega$ be a modulus of continuity such that
\[\omega_{\ast}(x) \leq \textnormal{const }\omega(x),\;\; x \geq 0.\]
Then for an arbitrary function $f \in \Lambda_{\omega}(\mathbb{R})$ and for arbitrary self-adjoint operators $A$ and $B$,
the following inequality holds for all $l$ and for all $j$, $1\leq j\leq l :$
\begin{equation*}
   s_j(f(A)-f(B))\leq \textnormal{const }
   \omega\big((1+j)^{-\frac{1}{p}}\Vert A-B \Vert_{S_{p}^l}\big)
    \Vert f \Vert_{\Lambda_{\omega}}.
  \end{equation*}
\end{corollary}

Let $H$, $\mathcal{H}$ be the Hankel operators defined in \cite{B}.

\begin{theorem}
Let $\omega$ be a modulus of continuity on $\mathbb{T}$. There exist unitary operators
$U$, $V$ and a real function $h$ in $\Lambda_{\omega_{\sharp}}(\mathbb(T))$  such that
\[
  \textnormal{rank}(U-V)=1 \qquad and\qquad  s_m(h(U)-h(V))\geq\omega\big((1+m)^{-1}\big).
\]
\end{theorem}

\begin{proof}
  Consider the operators $U$ and $V$ on space $L_2(\mathbb{T})$ with respect
  to the normalized Lebesgue measure on $\mathbb{T}$ defined by (see~\cite{B})
  \[
  Uf=\bar{z}f\; and \; Vf=\bar{z}f-2(f,\; 1)\bar{z},\; f \in L^2 .
  \]
  For $f\in C(\mathbb{T})$, we have
  \begin{equation*}
   \big((f(U)-f(V))z^j,z^k\big)= -2
   \begin{cases}
   \hat{f}(j-k), &\text{if } j\geq 0, k<0 ;\\
   \hat{f}(j-k), &\text{if } j<0, k\geq 0 ; \\
   0, &\text{otherwise}.
   \end{cases} 
  \end{equation*}
  Define function $g$ by
  \[
  g(\zeta)=\sum_{n=1}^{\infty}
  \omega(4^{-n})(\zeta^{4^{n}}+\bar{\zeta}^{4^{n}}),\qquad
  \zeta \in \mathbb{T}.  
  \]
  Then we have
  \begin{equation*}
   \Vert g*W_n \Vert_{L^{\infty}} \leq \text{const} \; \omega(2^{-n}) , \;
   \Vert g*W_n^{\sharp} \Vert_{L^{\infty}} \leq \text{const} \; \omega(2^{-n}) \; ,\; n\geq 0. 
     \end{equation*}
  Let $\xi$, $\eta$ be two arbitrarily different fixed points on $\mathbb{T}$, choose $N \geq 0$ such that
  $\frac{1}{2} \leq \frac{2^{-N}}{\abs{ \xi- \eta }} \leq 1$, then
  
  \begin{align*}
   \abs{g(\xi)-g(\eta)}
   &\leq \sum_{n=0}^{N} \abs{g_n(\xi)-g_n(\eta)}+\abs{(g-g*V_N)(\xi)-(g-g*V_N)(\eta)}\\
   & \leq \sum_{n=0}^{N} \abs{g_n(\xi)-g_n(\eta)}+2 \sum_{n=N+1}^{\infty}\Vert g_n \Vert_{L^{\infty}} \\
   & \leq \text{const} \sum_{n=0}^{N}2^n \Vert g_n \Vert_{L^{\infty}}\abs{ \xi- \eta }+2 \sum_{n=N+1}^{\infty}\Vert g_n \Vert_{L^{\infty}} \\
   & \leq \text{const} \sum_{n=0}^{N}2^n \omega(2^{-n})\abs{ \xi- \eta }+\text{const}\sum_{n=N+1}^{\infty}\omega(2^{-n})\\
   & \leq \text{const}\; \omega_{\ast}(\abs{ \xi- \eta })+\text{const}\int_{0}^{2^{-N}}\frac{\omega(t)}{t}dt\\
   & \leq \text{const} \; \omega_{\sharp}(\abs{ \xi- \eta }).
 \end{align*}
%  (see~\cite{A}, Theorem 7.1)

  Consider the matrix $\Gamma_{g}=\{\hat{g}(-j-k)\}_{j\geq 1,k\geq 0}=\{\hat{g}(j+k)\}_{j\geq 1,k\geq 0}$.\\
  Let $n\geq 1$. Define matrix
   $T_n=\{\hat{g}(j+k+4^{n-1}+1)\}_{0\leq j,k\leq 3\cdot 4^{n-1}}$,
   then
   \begin{equation*}
      T_n=
      \begin{bmatrix}
       {}           &   {}      &       {}            & \omega(4^{-n})  \\
       {}           &   {}      &  \omega(4^{-n})    & {}               \\
       {}           &   \iddots  &       {}            & {}                 \\
   \omega(4^{-n})  &   {}      &       {}            & {}  
      \end{bmatrix}.
   \end{equation*} 
   If $R$ is any matrix with the same size of $T_n$ such that $\text{rank}(R)<3\cdot 4^{n-1}$, then
   $\Vert T_n-R \Vert \geq \omega(4^{-n})$. It follows that $s_j(T_n)\geq\omega(4^{-n})$ for $j<3\cdot 4^{n-1}$.  
   For each $T_n$, there is some orthogonal projection $P_n$ such that $T_n=P_n\Gamma_gP_n$, hence
   $s_j(\Gamma_g)\geq s_j(T_n)\geq \omega(4^{-n})$ for all $n$ and for all $j$, $j<3\cdot 4^{n-1}$. Thus for all $j\geq 0$, we have
   \[
   s_j(\Gamma_g)\geq \omega\big(\frac{3}{16}\cdot (j+1)^{-1}\big)
   \geq \frac{3}{32}\cdot\omega \big((j+1)^{-1}\big).
   \]
   To complete the proof, it suffices to take $h=\frac{32}{3}g$.
 \end{proof}
\begin{corollary}
Let $\omega$ be a modulus of continuity such that
\[\omega_{\sharp}(x) \leq \textnormal{const }\omega(x),\;\; 0 \leq x \leq 2.\]
There exist unitary operators
$U$, $V$ and a real function $h$ in $\Lambda_\omega\mathbb(T)$  such that
\[
  \textnormal{rank}(U-V)=1 \qquad and\qquad  s_m(h(U)-h(V))\geq\omega\big((1+m)^{-1}\big).
\]
\end{corollary}
 
\begin{theorem}
Let $\omega$ be a modulus of continuity on $\mathbb{T}$ and $f$ be a 
continuous function on $\mathbb{T}$. If for all unitary operators
$U$ and $V$, we have
\begin{equation*}
   s_n(f(U)-f(V))\leq \textnormal{const } 
   \omega\big((1+n)^{-\frac{1}{p}}\Vert U-V \Vert_{S_{p}}\big),
 \text{  for all }n \geq 0\, ,
\end{equation*}
then $f \in \Lambda_{\omega}(\mathbb{T}).$
\label{Thm:upperIneUni}
\end{theorem}

\begin{proof}
  Let $\zeta, \eta \in \mathbb{T}$, we can select commuting unitary operators
  $U$ and $V$ such that $s_0(U-V)=s_1(U-V)=\ldots=s_n(U-V)=\abs{\zeta-\eta}$ 
  and $s_k(U-V)=0,\, k\geq n+1$. Then $ s_n(f(U)-f(V))=\abs{f(\zeta)-f(\eta)}$, 
  $\Vert U-V \Vert_{S_{p}}=(1+n)^{\frac{1}{p}}\cdot \abs{\zeta-\eta}$. 
\end{proof}

\begin{theorem}
 Let $\omega$ be a modulus of continuity on $\mathbb{R}$ and $f$ be a 
continuous function on $\mathbb{R}$. If for all self-adjoint operators
$A$ and $B$, we have
\begin{equation*}
   s_n(f(A)-f(B))\leq \textnormal{const } 
   \omega\big((1+n)^{-\frac{1}{p}}\Vert A-B \Vert_{S_{p}}\big),
 \text{  for all }n \geq 0\, ,
\end{equation*}
then $f \in \Lambda_{\omega}(\mathbb{R}).$
\end{theorem}
\begin{proof}
Similar to Theorem \ref{Thm:upperIneUni}.
\end{proof}

\begin{theorem}
 Let $\omega$ be a modulus of continuity over $\mathbb{R}$. There exist 
 self-adjoint operators $A$, $B$, and a real function $f$ 
 in $\Lambda_{\omega_{\sharp}}(\mathbb{R})$ such that 
 \[
 \textnormal{rank}(A-B)=1 \;\, and\;\,  
 s_m(f(A)-f(B))\geq\omega\big((1+m)^{-1}\big)\text{, for all } m \geq 0.
 \]   
\end{theorem}

\begin{proof}
 WLOG, we assume $\omega(t)=\omega(2)$, for all $t\geq 2$, that is, $\omega$ 
 can be regarded as a modulus of continuity on $\mathbb{T}.$ \\
 We then choose a function(see~\cite{B}, Lemma 9.6) $\rho \in C^{\infty}(\mathbb{T})$ such that 
 $\rho(\zeta)+\rho(i\zeta)=1$, $\rho(\zeta)=\rho(\bar{\zeta})$ for all $\zeta \in 
 \mathbb{T}$, and $\rho$ vanishes in a neighborhood of the set $\{-1, 1\}.$ Note 
 that $\rho \in \Lambda_{\omega}(\mathbb{T}),$ since $\omega(st)
 \geq \frac{s}{2}\omega(t),$ for all $t\geq 0$ and $s$, $0<s<1.$\\
 Define function $g_1$ by 
 \[
 g_1(\zeta)=\sum_{n=1}^{\infty}\omega(4^{-n})(\zeta^{4^{n}}+\bar{\zeta}^{4^{n}}),
 \qquad \zeta \in \mathbb{T}. 
 \]  
 Then $g_1 \in \Lambda_{\omega_{\sharp}}(\mathbb{T}).$ If $g_0\myeq C\rho 
 g_1$ for a sufficient large number $C$, then $g_0 \in \Lambda_{\omega_{\sharp}}(\mathbb{T}),$ 
 vanishes in a neighborhood of the set $\{-1,1\}$ and $g_0(\zeta)=g_0(\bar{\zeta})$ for all $\zeta \in 
 \mathbb{T}$, and $s_m(H_{g_{0}})\geq \omega\big((1+m)^{-1}\big)$ for all $m\geq 0.$\\
 Define $\varphi(x)=(x^2+1)^{-1}$(as in ~\cite{B}, Theorem 9.9), then there exists
 a compactly supported real bounded function $f$ such that
  $f(\varphi(x))=g_0(\frac{x-i}{x+i})$ and a simple calculation shows that $f$ belongs
  to $\Lambda_{\omega_{\sharp}}(\mathbb{R}).$ Denote $L_{e}^{2}(\mathbb{R})$ the subspace of even functions in $L^{2}(\mathbb{R})$. Consider operators $A$ and $B$ on
  $L_{e}^{2}(\mathbb{R})$ defined by $A(g)=\mathbf{H}^{-1}M_{\varphi}\mathbf{H}(g)$ and 
  $B(g)=\varphi g,$ here $\mathbf{H}$ is the Hilbert transform defined on $L^2(\mathbb{R})$ ( see \cite{B}) and $M_{\varphi}$ 
  is the multiplication by $\varphi$. Then
  $\text{rank}(A-B)=1,$ and we have
  \begin{equation*}
  s_m(f(B)-f(A))\geq \sqrt{2}s_m(\mathcal{H}_{f\circ\varphi})=\sqrt{2}s_m(H_{g_0})
  \geq \sqrt{2}\omega\big((1+m)^{-1}\big). \qedhere
  \end{equation*}
  \end{proof}
\section{Estimates for other types of operators}
%  The following estimates are given without proofs in case of contractions(see \cite{B}, Theorem 7.1), maximal dissipative operators(see \cite{D}, Theorem %7.1 and 7.2), normal operators(see \cite{E}, Theorem 8.1 and 8.2) and $n$-tuples of commuting self-adjoint operators(see \cite{G}, Theorem 6.1 and 6.2).
%\label{sec:num3}
  The following estimates are given without proofs in case of contractions, maximal dissipative operators, normal operators and $n$-tuples of commuting self-adjoint operators.
\label{sec:num3}
\begin{theorem}
There exists a constant $c>0$ such that for every modulus of continuity $\omega$, every $f$ in
$\big(\Lambda_{\omega}(\mathbb{R})\big)_{+}$ and for arbitrary contractions $T$ and $R$ on
Hilbert space,
the following inequality holds for all $l$ and for all $j$, $1\leq j\leq l :$
\begin{equation*}
   s_j(f(T)-f(R))\leq c\;
   \omega_{\ast}\big((1+j)^{-\frac{1}{p}}\Vert T-R \Vert_{S_{p}^l}\big)
    \Vert f \Vert_{\Lambda_{\omega}}.
   \label{Ine:svalue_cntrctn}
\end{equation*}
\label{Thm:mainthm_cntrctn}
\end{theorem} 
To prove this result, the following result is important(see \cite{A}, \cite{B} and \cite{L}):

There exists a constant $c$ such that for arbitrary trigonometric polynomial $f$ of degree $n$ and for
arbitrary contractions $T$ and $R$ on Hilbert space,
 
\[
   \Vert(f(T)-f(R)\Vert_{S_{p}}\leq c\;
   n\Vert f \Vert_{L^{\infty}}\Vert T-R \Vert_{S_{p}}.
\]

Denote $\mathcal{F}$ the \textit{Fourier transform} on $L_1(\mathbb{R}^{n})$, $n \geq 1$ by:
\begin{equation*}
(\mathcal{F}f)(t)=\int_{\mathbb{R^{n}}}f(x)e^{-i(x,t)}dx, \text{where}
\end{equation*} 
\begin{equation*}
 x=(x_1, ..., x_n), \; t=(t_1, ..., t_n), \;(x,t)\myeq x_1t_1+...+x_nt_n.
\end{equation*} 
\begin{theorem}
 There exists a constant $c>0$ such that for every modulus of continuity $\omega$, every $f$ in
$\big(\Lambda_{\omega}(\mathbb{R})\big)_{+}$ and for arbitrary maximal dissipative operators $L$ and $M$
 with bounded difference,
the following inequality holds for all $l$ and for all $j$, $1\leq j\leq l :$
\begin{equation*}
   s_j(f(L)-f(M))\leq c\;
   \omega_{\ast}\big((1+j)^{-\frac{1}{p}}\Vert L-M \Vert_{S_{p}^l}\big)
    \Vert f \Vert_{\Lambda_{\omega}}.
   \label{Ine:svalue_dissi}
\end{equation*}
\label{Thm:mainthm_dissi}
\end{theorem} 

To prove this result, the following result is important(see \cite{D}):

There exists a constant $c>0$ such that for every  function
$f$ in $H^{\infty}(\mathbb{C}_{+})$ with $\text{supp }\mathcal{F}f \subset [0, \sigma]$, $\sigma > 0, $ and for arbitrary maximal dissipative operators $L$ and $M$
 with bounded difference,
 
\[
   \Vert f(L)-f(M)\Vert_{S_p}\leq c\, \sigma\Vert f \Vert_{L^{\infty}}\Vert L-M \Vert_{S_p}.
\]
 
\begin{theorem}
There exists a constant $c>0$ such that for every modulus of continuity $\omega$, every $f$ in
$\Lambda_{\omega}(\mathbb{R}^2)$ and for arbitrary normal operators $N_1$ and $N_2$,
the following inequality holds for all $l$ and for all $j$, $1\leq j\leq l :$
\begin{equation*}
   s_j(f(N_1)-f(N_2))\leq c\;
   \omega_{\ast}\big((1+j)^{-\frac{1}{p}}\Vert N_1-N_2 \Vert_{S_{p}^l}\big)
    \Vert f \Vert_{\Lambda_{\omega}}.
   \label{Ine:svalue_nor}
\end{equation*}
\label{Thm:mainthm_nor}
\end{theorem} 
To prove this result, the following result is important(see \cite{E}):

There exists a constant $c>0$ such that for every bounded continuous function $f$ on $\mathbb{R}^{2}$ with
 \[
 \text{supp }\mathcal{F}f \subset \{\zeta\in\mathbb{C}:\abs{\zeta}\leq\sigma\}, \;\; \sigma>0,
 \]
and for arbitrary normal operators $N_1$ and $N_2$,
\[
   \Vert(f(N_1)-f(N_2)\Vert_{S_{p}}\leq c\;
  \sigma\Vert f \Vert_{L^{\infty}}\Vert N_1-N_2 \Vert_{S_{p}}.
\]
 
\begin{theorem}
Let n be a positive integer and $p \geq 1$. There exists a positive number $c_n$ such that for every modulus of continuity $\omega$, every $f$ in
$\Lambda_{\omega}(\mathbb{R}^n)$ and for arbitrary $n$-tuples of commuting self-adjoint operators
 $(A_1, ..., A_n)$ and $(B_1, ..., B_n)$,
the following inequality holds for all $l$ and for all $j$, $1\leq j\leq l :$
\begin{equation*}
   s_j(f(A_1, ..., A_n)-f(B_1, ..., B_n))\leq c_n
   \max_{1 \leq j \leq n}
   \omega_{\ast}\big((1+j)^{-\frac{1}{p}}\Vert A_j-B_j \Vert_{S_{p}^l}\big)
    \Vert f \Vert_{\Lambda_{\omega}}.
   \label{Ine:svalue_tuples}
\end{equation*}
\label{Thm:mainthm_tuples}
\end{theorem}  
To prove this result, the following result is important(see \cite{G}):

There exists a constant $c_n>0$ such that for every bounded continuous function $f$ on $\mathbb{R}^{n}$ with
 \[
 \text{supp }\mathcal{F}f \subset \{\xi\in\mathbb{R}^{n}:\abs{\xi}\leq\sigma\}, \;\; \sigma>0,
 \]
and for arbitrary $n$-tuples of commuting self-adjoint operators
 $(A_1, ..., A_n)$ and $(B_1, ..., B_n)$,
\[
   \Vert f(A_1, ..., A_n)-f(B_1, ..., B_n)\Vert_{S_{p}}\leq c_n\;
  \sigma\Vert f \Vert_{L^{\infty}}\max_{1 \leq j \leq n}\Vert  A_j-B_j \Vert_{S_{p}}.
\]

\bibliographystyle{amsplain}

\begin{thebibliography}{99}

\bibitem {A} A.B. Aleksandrov, V.V. Peller, Operator H\"older-Zygmund functions, Adv. Math.
\textbf{224} (2010), 910--966.

\bibitem {B} A.B. Aleksandrov, V.V. Peller, Functions
of operator under perturbation of
class $S_p$, J.Func.Anal. 
\textbf{258} (2010), 3675--3724.
\bibitem {C} A.B. Aleksandrov, V.V. Peller, Functions
of perturbed unbounded self-adjoint operators. Operator Bernstein type inequalities, Indiana Univ. Math. J. 
\textbf{59:4} (2010), 1451--1490.

\bibitem {D} A.B. Aleksandrov, V.V. Peller, Functions of perturbed dissipative operators, Algebra i Analiz
\textbf{23} (2011), 9--51; translation in St. Petersburg Math. J. \textbf{23} (2012), 209--238.

\bibitem {E} A.B. Aleksandrov, V.V. Peller, D. Potapov and F. Sukochev, Functions of normal operators under
 perturbations, Adv. Math.
\textbf{226} (2011), 5216--5251. 


\bibitem {F} A. Zygmund, Trigonometric series, 2nd ed. Vols. I, II. Cambridge University Press, New York, 1959.

\bibitem {G} F.L. Nazarov, V.V. Peller, Functions of $n$-tuples of commuting self-adjoint
 operators, J. Funct.Anal.
\textbf{266} (2014), 5398--5428. 

\bibitem {H} M.S. Birman, M.Z. Solomyak, Double Stieltjes operator integrals. III, 
Problems of Math. Phys., Leningrad. Univ. 
\textbf{6} (1973), 27--53 (Russian).
\bibitem {I} M.S. Birman, M.Z. Solomyak, Spectral theory of
 selfadjoint operators in Hilbert spaces, Mathematics and its Applications (Soviet Series), D. Reidel Publishing Co., Dordrecht, 1987.
\bibitem {J} V.V. Peller, Hankel operators of class $S_p$ and 
their applications (rational approximation, Gaussian processes, the problem of majorizing operators), 
Mat. Sbornik, 
\textbf{113} (1980), 538--581 (Russian). \par
English Transl. in Math. USSR Sbornik, \textbf{41} (1982), 443-479.

\bibitem {K} V.V. Peller, Hankel operators in theory of perturbations of unitary and self-adjoint operators, 
Funktsional. Anal.i Prilozhen. 
\textbf{19:2} (1985), 37--51 (Russian). \par
English Transl. in Funct. Anal. Appl. \textbf{19} (1985), 111-123.

\bibitem {L} V.V. Peller, For which $f$ does $A-B \in S_p$ imply that $f(A)-f(B) \in S_p ?$, 
Operator Theory, Birkhauser, 
\textbf{24} (1987), 289--294.

\bibitem {M} V.V. Peller, Hankel operators in the perturbation theory of unbounded 
self-adjoint operators. Analysis and partial differential equations, 529-544, Lecture Notes in Pure and Appl. Math., 
\textbf{122}, Dekker, New York, 1990.
\end{thebibliography}
%\section*{Referrence}
%\label{sec:referr}
%\addcontentsline{toc}{section}{\nameref{sec:referr}}

\end{document}